\documentclass[11pt]{amsart}
\usepackage[pagewise]{lineno}
\usepackage{amsmath,amssymb,mathrsfs,color}
\def\R{{\mathbb R}}

\newtheorem{thm}{Theorem}[section]
\newtheorem{cor}[thm]{Corollary}

\newtheorem{prop}[thm]{Proposition}

\theoremstyle{definition}
\newtheorem{de}[thm]{Definition}
\theoremstyle{remark}
\newtheorem{rem}[thm]{Remark}
\newtheorem{exam}[thm]{Example}
\numberwithin{equation}{section}

\allowdisplaybreaks
\usepackage{amsmath}

\begin{document}

\title[Continuous dependence for MVSDEs]
{Continuous dependence for McKean-Vlasov SDEs under distribution-dependent Lyapunov conditions}

\author{Jun Ma}
\address{J. Ma: School of Mathematics and Statistics, and Center for Mathematics and Interdisciplinary Sciences,
Northeast Normal University, Changchun, 130024, P. R. China}
\email{mathmajun@163.com}
\author{Zhenxin Liu}


\address{Z. Liu (Corresponding author): School of Mathematical Sciences,
Dalian University of Technology, Dalian 116024, P. R. China}
\email{zxliu@dlut.edu.cn}

\date{}



\begin{abstract}
In this paper, we consider the continuous dependence on initial values and parameters
of solutions as well as invariant measures for McKean-Vlasov SDEs under distribution-dependent Lyapunov conditions.
In contrast to the classical SDEs, the solutions for McKean-Vlasov SDEs do not converge in probability
although the initial values converge in probability,
which is due to the mismatch of the distances between measures.
Finally, we give some examples to illustrate our theoretical results.
\end{abstract}

\maketitle

\section{Introduction}
A signature of structural stability in dynamical systems is the continuous dependence of solutions and invariant measures
on initial values and parameters.
Investigating this problem enables us to understand deeply the robustness
and global bifurcation phenomenon of the dynamical systems in consideration,
which, in addition, is of great significance in practical applications.
In this paper, we prove the continuous dependence of solutions and invariant measures for McKean-Vlasov SDEs (MVSDEs):
\begin{equation}\label{main}
dX_t= b(t,X_t,\mathcal{L}_{X_t})dt+ \sigma(t,X_t,\mathcal{L}_{X_t})dW_t,
\end{equation}
where $\mathcal{L}_{X_t}$ denotes the law of $X_t$.
Due to the dependence of coefficients on distribution, MVSDEs are also known as distribution-dependent SDEs.
In addition, MVSDEs are called mean-field SDEs since they are the limit of $N$-particle systems
if the coefficients satisfy some conditions:
\begin{align}
dX^N_{i,t}= b(t,X^N_{i,t},\mu^N_t)dt+ \sigma(t,X^N_{i,t},\mu^N_t)dW_t,\nonumber
\end{align}
for each $i=1,2,...,N$,
where $\mu^N_t:=\frac 1N \sum_{j=1}^N\delta_{X^N_{j,t}}$;
for details, see Sznitman \cite{Sznitman1,Sznitman2} for instance.

With the increasing demands on practical financial markets and social systems,
MVSDEs have drawn much attention.
For example, the change rate of prices in financial markets may depend on the macrocosmic distribution.
Inspired by Kac's work on the Vlasov kinetic equation \cite{Kac},
MVSDEs were first studied by McKean \cite{Mckean} which showed the propagation of chaos in
physical systems of N-interacting particles related to the Boltzmann equation.
In order to study large population deterministic and stochastic differential games,
Lasry, Lions \cite{Larsy-Lions1,Larsy-Lions2,Larsy-Lions3} introduced
mean-field games, which were independent of the work of Huang, Malhame and Caines \cite{HMC1,HMC2}.
Owing to their work, MVSDEs are studied more extensively.
Wang \cite{Wang_18}, Ren et al. \cite{RTW} as well as Liu and Ma \cite{LM1,LM2}
obtained the existence and uniqueness under different conditions.
Butkovsky \cite{Butkovsky}, Eberle et al. \cite{EGZ}, Bogachev et al. \cite{BRS}, Wang \cite{Wang_18} as well as Liu and Ma \cite{LM1,LM2}
showed ergodicity under different conditions.
Buckdahn et al. \cite{BLPR} investigated the relationship between the functional and the associated second-order PDE.
Moreover, Cannarsa et al. \cite{CCCW} find solutions by weak KAM approach, which was also utilized by Iturriaga and Wang \cite{IW} to study the asymptotic behavior for first-order mean-field games.

In this paper, we are interested in the continuous dependence of solutions and invariant measures
on initial values and parameters for MVSDEs under distribution-dependent Lyapunov conditions.
Until now, we find that there are few study about this issue.
Bahlali et al. \cite{BMM} proved the convergence of solutions in $L^2$ uniformly in $[0,T]$ under Lipschitz and linear growth conditions,
where the linear growth condition is linear in space variable and bounded in distribution variable.
Qiao \cite{Qiao} showed the convergence in $L^2$ of the corresponding solutions for multivalued MVSDEs
and the convergence in $W_1$-Wasserstein distance of the corresponding invariant measures
under monotone and linear growth conditions.
Wang \cite{Wang_18} proved the continuous dependence of solutions on determinant initial values under monotone and linear growth conditions.
Hammersley et al. \cite{HSS} investigated the continuous dependence on initial values under a Lyapunov condition, which is different from ours.
Wu et al. \cite{WXZ} obtained the continuous dependence on initial values for McKean-Vlasov stochastic functional differential equations.

Although there are few study about continuous dependence for MVSDEs, there exist some results for the classical SDEs.
In general, there are three ways to illustrate the convergence of the corresponding solutions,
i.e. convergence in $L^2$, probability and distribution.
Friedman \cite{F} showed that the corresponding solutions varied continuously in $L^2$ if the coefficients converge on any compact set.
Gihman and Skorohod \cite{GS} studied the convergence in $L^2$ and probability.
Da Prato and Tudor \cite{DT} proved the continuous dependence in $L^2$ uniformly on $[0,T]$, probability and distribution for semi-linear stochastic partial differential equations (SPDEs) when the coefficients converge point-wise and initial values converges in the above three ways.
Qiu and Wang \cite{QW} get the similar results of recurrent solutions.
Cheng and Liu \cite{CL} illustrated similar results for SPDEs.
In summary, the solutions inherit the convergence of initial values
if the coefficients converge point-wise or on any compact set.
Thus, we want to know whether these results also hold for MVSDEs.

In this paper, we mainly investigate the continuous dependence
of solutions and invariant measures for MVSDEs under distribution-dependent Lyapunov conditions.
To be specific, we study that $X_{k,t}$ converges to $X_t$ in $W_2$-Wasserstein distance uniformly on $[0,T]$
under some suitable conditions (for details see Section 3),
where $X_{k,t}$ and $X_t$ are solutions of the following MVSDEs, respectively:
\begin{align}
\left\{
\begin{aligned}\label{pMVSDE}
dX_{k,t}&= b_k(t,X_{k,t},\mathcal{L}_{X_{k,t}})dt+ \sigma_k(t,X_{k,t},\mathcal{L}_{X_{k,t}})dW_t,\\
X_{k,0}&= \xi_k,
\end{aligned}
\right.
\end{align}
where $k=1,2,...$,
and
\begin{align}
\left\{
\begin{aligned}\label{MVSDE}
dX_t&= b(t,X_t,\mathcal{L}_{X_t})dt+ \sigma(t,X_t,\mathcal{L}_{X_t})dW_t,\\
X_0&= \xi.
\end{aligned}
\right.
\end{align}
In order to illustrate this result, we utilize Skorokhod's representation theorem and martingale representation theorem.
Moreover, the invariant measure of equation \eqref{pMVSDE} converges weakly to the invariant measure of equation \eqref{MVSDE}
if the coefficients $b,\sigma, b_k,\sigma_k$ are independent of $t$,
which is obtained by the same method as above.
However, in contrast to classical SDEs, we do not have the result about the convergence in probability when initial values converge in probability,
since the coefficients, which depend on the distribution, are Lipschitz continuous with respect to the distribution variables in $W_2$-Wasserstein distance,
which gives rise to the mismatch of distance between measures.
To illustrate this situation, we give a counterexample in which the coefficients are Lipschitz.
Meanwhile, for the completeness of this article, we also show that solutions are convergent in $L^2$ uniformly on $[0,T]$,
and invariant measures are convergent in $W_2$ under Lipschitz conditions
when the coefficients converge point-wise.

The rest of this paper is arranged as follows.
In Section 2, we collect a number of preliminary results
concerning the existence and uniqueness of solutions as well as invariant measures for MVSDEs and so on.
Section 3 presents the continuous dependence of solutions and invariant measures on initial values and parameters
under distribution-dependent Lyapunov conditions,
and give a counterexample which illustrates that the solution does not inherit its convergence if the initial values converge in probability.
In Section 4, we provide some examples to illustrate our theoretical results.
In Appendix, we give the proof of continuous dependence under Lipschitz conditions.

\section{Preliminaries}
Throughout the paper, let $(\Omega, \mathcal F, \{\mathcal F_t\}_{t\ge 0},P)$ be a filtered complete probability space.
Assume that the filtration $\{\mathcal F_t\}$ satisfies the usual condition,
i.e. it is right continuous and $\mathcal F_0$ contains all $P$-null sets.
Let $W$ be an $n$-dimensional Brownian motion defined on $(\Omega, \mathcal F, \{\mathcal F_t\}_{t\ge 0}, P)$.
We denote by $A^\top$ the transpose of matrix $A\in \mathbb R^{n\times m}$ with $n,m\geq1$,
$tr(A)$ the trace of $A$ and $|A|:=\sqrt {tr(A^\top A)}$ the norm of $A$.
Assume that the coefficients $b,b_k:[0,\infty) \times \mathbb R^d \times \mathcal P (\mathbb R^d)\to \mathbb R^d$, $\sigma,\sigma_k:[0,\infty) \times \mathbb R^d \times \mathcal P (\mathbb R ^d)\to \mathbb R^{d\times n}$, and
$X_0,X_{k,0}$ are $\mathcal F_0$-measurable and satisfy some integrable condition to be specified below.

Let us introduce some basic notations as follows. Let $\mathcal P(\mathbb R^d)$ be the space of probability measures on $\mathbb R^d$.
We use the usual Wasserstein distance $W_p$ on $\mathcal P_p(\R^d)$ with $p=1,2$ in what follows, i.e.
\begin{equation}
W_p(\mu, \nu):=\inf_{\pi \in \mathcal C(\mu, \nu)} \bigg[\int_{\mathbb R^d\times\R^d} |x-y|^p \pi(dx, dy)\bigg]^{1/p}\nonumber
\end{equation}
for $\mu, \nu\in \mathcal P_p(\R^d)$,
where $\mathcal P_p(\R^d):=\{\mu\in\mathcal P(\R^d): \int_{\R^d}|x|^p \mu(dx)<\infty \}$,
and $\mathcal C(\mu,\nu)$ denotes the set of all coupling between $\mu$ and $\nu$.
We also define Wasserstein distance $\bar W_2$ on $\mathcal P(C[0,T])$ by
\begin{align}\nonumber
\bar W_2(\mu,\nu):=\inf_{\pi\in\mathcal C(\mu,\nu)}\bigg[\int_{C[0,T]\times C[0,T]}|x-y|^2\pi(dx,dy)\bigg]^{1/2},
\end{align}
for $\mu,\nu\in\mathcal P(C[0,T])$,
where $\mathcal P(C[0,T])$ denotes the set of probability measures on $C[0,T]$.
As usual, we also denote $\mu(f):=\int_{\R^d} f(x)\mu(dx)$ in what follows for any function $f$ defined on $\R^d$ and $\mu\in\mathcal P(\R^d)$.

We introduce two propositions which show the existence and uniqueness of solutions as well as invariant measures
for MVSDEs under distribution-dependent Lyapunov conditions; for details see \cite{LM2}.
Before giving these results, we introduce some definitions about Lions derivative and differentiable space;
for details see Cardaliaguet \cite{Cardaliaguet} and Chassagneux et al. \cite{CCD} for instance.
\begin{de}[Lions derivative]
A function $f:\mathcal P_2(\mathbb R^d)\to\mathbb R$ is called differentiable at $\mu\in\mathcal P_2(\mathbb R^d)$, denoted by $\partial_{\mu} f$,
if there exists a random variable $X\in L^2(\Omega)$ with $\mu=\mathcal L_X$
such that $F(X):=f(\mathcal L_X)$ and $F$ is Fr\'echet differentiable at $X$.
$f$ is called differentiable on $\mathcal P_2(\mathbb R^d)$ if $f$ is differentiable at any $\mu\in\mathcal P_2(\mathbb R^d)$.
\end{de}
\begin{de}
The space $C^{(1,1)}(\mathcal P_2(\mathbb R^d))$ contains the functions $f:\mathcal P_2(\mathbb R^d)\to\mathbb R$
satisfying the following conditions:
(i) $f$ is differentiable and its derivative $\partial_{\mu} f(\mu)(y)$ has a jointly continuous version in $(\mu,y)$, still denoted $\partial_{\mu} f(\mu)(y)$;
(ii) $\partial_{\mu} f(\mu)(\cdot)$ is continuously differentiable for any $\mu$,
and its derivative $\partial_y \partial_{\mu} f(\mu)(y)$ is jointly continuous at any $(\mu,y)$.
\end{de}
\begin{enumerate}

\item [(H1)]For any $N\geq 1$, there exists a constant $ C_N\geq 0 $ such that for any $ |x|, |y|\leq N$ and  ${\rm supp} \mu, {\rm supp} \nu \subset B(0,N)$ we have
\begin{align*}
&|b(t, x, \mu) |+ |\sigma(t, x, \mu)| \leq C_N,\\
&| b(t, x, \mu)- b(t, y, \nu) |+ |\sigma(t, x, \mu)- \sigma(t, y, \nu)| \leq  C_N\big(| x- y | + W_2(\mu,\nu)\big).
\end{align*}
Here $B(0,N)$ denotes the closed ball in $\R^d$ centered at the origin with radius $N$.

\item [(H2)](Lyapunov condition)
There exists a nonnegative function $ V\in C^{2,(1,1)}(\mathbb R^d \times \mathcal P_2(\mathbb R^d))$ such that there exist nonnegative constants $ \lambda \in \mathbb R^+ $ satisfying for all $(t,x,\mu)\in \mathbb R^+\times\mathbb R^d \times \mathcal P_2(\R^d)$
\begin{align*}
(\mathcal LV)(t,x,\mu)  &\leq  \lambda V(x, \mu), \\
V_R(\mu)&:= \inf_{|x|\geq R}V(x, \mu)\to \infty ~as~ R\to \infty,
\end{align*}
where
\begin{align*}
C^{2,(1,1)}(\mathbb R^d\times \mathcal P_2(\mathbb R^d))
:={}&\{f:\mathbb R^d\times\mathcal P_2(\mathbb R^d)\to\mathbb R|
f(\cdot, \mu)\in C^2(\mathbb R^d) \quad \hbox{for }\mu\in\mathcal P_2(\mathbb R^d),\\
&\quad f(x,\cdot)\in C^{(1,1)}(\mathcal P_2(\mathbb R^d))\quad \hbox{for }x\in \mathbb R^d\},
\end{align*}
and
\begin{align*}
(\mathcal LV)(t,x,\mu)
:={}&
b(t,x,\mu)\cdot \partial_x V(x,\mu)+ \frac 12tr((\sigma\sigma^{\top})(t,x,\mu)\cdot\partial_x^2 V(x,\mu))\\
&+{}\int \big[b(t,y,\mu)\cdot\partial_{\mu} V(x,\mu)(y)
 + \frac 12tr((\sigma\sigma^{\top})(t,y,\mu)\cdot\partial_y \partial_{\mu} V(x,\mu)(y)\big]\mu(dy).
\end{align*}

\item [(H2')](Lyapunov condition)
There exists a nonnegative function $V\in C^{2,(1,1)}(\mathbb R^d \times \mathcal P_2(\mathbb R^d))$ such that there is a constant $\gamma\ge 0$ satisfying
\begin{align*}
(\mathcal LV)(x,\mu)  &\leq{} -\gamma, \\
V_R&:={} \inf_{|x|\vee|\mu|_2\geq R}V(x, \mu)\to \infty ~as~ R\to \infty,
\end{align*}
for all $(x,\mu)\in \mathbb R^d \times \mathcal P_2(\mathbb R^d)$,
where $|\mu|_2^2:=\int_{\mathbb R^d}|x|^2\mu(dx)$.

\item [(H3)](Continuity) For any bounded sequences $ \{{x_n, \mu_n}\} \in \mathbb R^d \times \mathcal P_V (\mathbb R^d) $ with $ x_n \to x $ and $\mu_n \to \mu$ weakly in $\mathcal P (\mathbb R^d)$ as $ n\to \infty $, we have
\begin{equation}\nonumber
\lim_{n\to \infty} \sup_{t\in [0,T]}{ | b(t, x_n, \mu _n)- b(t, x, \mu) |+ |\sigma(t, x_n, \mu _n)- \sigma(t, x, \mu)|}= 0.
\end{equation}
where $ \mathcal P_V (\mathbb R^d):= \left\{ \mu \in \mathcal P(\mathbb R^d):\int_{\mathbb R^d} V(x,\mu)\mu(dx) <\infty\right\}$.

\item[(H4)] There exist constants $\ell>1,K>0$ such that
\begin{align*}
|b(t,x,\mu)|^{2\ell}+|\sigma(t,x,\mu)|^{2\ell} \le K(1+V(x,\mu)).
\end{align*}

\item [(H5)]There exist constants $M,\epsilon >0 $ and increasing unbounded function $L: \mathbb N \to (0,\infty)$ such that for any $n\geq 1,x,y\in C[0,T],|x(t)|\vee |y(t)| \leq n$ and $\mu,\nu \in \mathcal P (C[0,T])$ satisfying
\begin{align*}
&| b(t, x_t, \mu_t)- b(t, y_t, \nu_t) |+ |\sigma(t,x_t, \mu_t)- \sigma(t, y_t, \nu_t)|\\
&\leq L_n(|x_t- y_t|+ W_{2,n}(\mu_t , \nu_t)+ Me^{-\epsilon L_n}(1\wedge W_2(\mu_t,\nu_t))),
\end{align*}
where
\begin{align*}
W_{2,n}^2(\mu_t,\nu_t)&:= \inf_{\pi \in \mathcal C(\mu,\nu)} \int_{C[0,T] \times C[0,T]}|x_{t\wedge \tau_x^n\wedge \tau_y^n} -y_{t\wedge \tau_x^n\wedge \tau_y^n}|^2 \pi(dx,dy),\\
\tau_x^n&:=\inf\{t\ge 0: |x_t|\ge n\},
\tau_y^n:=\inf\{t\ge 0: |y_t|\ge n\}.
\end{align*}

\end{enumerate}
\begin{prop}\cite[Theorem 3.2]{LM2}\label{prop1}
Assume (H1)--(H4). Then for any $T>0, X_0\in L^2(\Omega, \mathcal F, P)$, equation \eqref{MVSDE} has a solution $X_.$
which satisfies
\begin{equation}\label{ES}
EV(X_t,\mathcal L_{X_t})\le e^{\lambda t}EV(X_0,\mathcal L_{X_0}),\quad\hbox{for } t\in [0,T].
\end{equation}
Moreover, if (H5) holds, the solution is unique.
\end{prop}

\begin{prop}\cite[Theorem 4.3]{LM2}\label{prop2}
Assume that (H1), (H2'), (H3)-(H5) holds with time-independent coefficients, semigroup $P_t$ is Feller,
and that there is a function $C:\mathbb R^+\to\mathbb R^+$ satisfying $\sup_{t\in\mathbb R^+}C(t)<\infty$ such that
\begin{align}\nonumber
\omega(P_t^*\mu,P_t^*\nu)
\le{}
C(t)\omega(\mu,\nu) \quad \hbox{for }\mu,\nu\in\mathcal P_2(\mathbb R^d),
\end{align}
where $\omega$ is L\'evy-Prohorov distance and $P_t^*\mu$ denotes the distribution at time $t$ with initial distribution $\mu$.
Then there exists at least one invariant measure.
Moreover, if there exists a constant $t_0>0$ such that $C(t_0)<1$,
there is a unique invariant measure.
\end{prop}

\section{Continuous dependence}
In this section, we consider continuous dependence of solutions and invariant measures on initial values and parameters for MVSDEs
under distribution-dependent Lyapunov conditions. And we divide this section into three parts.

\subsection{Continuous dependence of solutions}
In this subsection, we prove that solutions vary continuously with respect to initial values and parameters
under distribution-dependent Lyapunov conditions,
which is obtained by Skorokhod's representation theorem and martingale representation theorem.
\begin{thm}\label{thm1}
Assume that the conditions in Proposition \ref{prop1} hold with $b,\sigma$ replaced by $b_k,\sigma_k$,
and
\begin{align}\label{Cond1}
\lim_{k\to\infty}EV(X_{k,0},\mathcal L_{X_{k,0}})=EV(X_{0},\mathcal L_{X_{0}}).
\end{align}
Assume further that
\begin{align}\label{CoeCon1}
\lim_{k\to \infty}|b_k(t, x, \mu)- b(t, x, \mu)|+ |\sigma_k(t, x, \mu)- \sigma(t, x, \mu)|=0 \quad  \mbox{point-wise}.
\end{align}
Then solutions vary continuously with respect to initial values and parameters.
To be precise, $$X_{k}\to X ~in ~W_2~on~C[0,T].$$
\end{thm}
\begin{proof}
By H\"older's inequality, BDG's inequality and (H4), there exists a constant $C(\ell,K)>0$ such that
\begin{align}\label{ESofUC}
& E(\sup_{t\in[s,(s+\epsilon)\wedge T]}|X_{k,t}- X_{k,s}|^{2\ell})\\\nonumber
={} &
E\bigg(\sup_{t\in[s,(s+\epsilon)\wedge T]}\bigg|\int_s^t b_k(r,X_{k,r},\mathcal{L}_{X_{k,r}})dr
+ \int_s^t \sigma_k(r,X_{k,r},\mathcal{L}_{X_{k,r}})dW_r \bigg|^{2\ell}\bigg)\\\nonumber
\leq{} &
C(\ell) E\int_s^{(s+\epsilon)\wedge T}|b_k(r,X_{k,r},\mathcal{L}_{X_{k,r}})|^{2\ell}dr\cdot \epsilon^{2\ell-1}\\\nonumber
& +{}
C(\ell)E\bigg(\int_s^{(s+\epsilon)\wedge T}|\sigma_k(r,X_{k,r},\mathcal{L}_{X_{k,r}})|^2 dr\bigg)^{\ell}\\\nonumber
\leq{} &
C(\ell,K) E\int_s^{(s+\epsilon)\wedge T}1+V(X_{k,r},\mathcal{L}_{X_{k,r}})dr\cdot \epsilon^{2\ell-1}\\\nonumber
& +{}
C(\ell)E\bigg(\int_s^{(s+\epsilon)\wedge T}|\sigma_k(r,X_{k,r},\mathcal{L}_{X_{k,r}})|^{2\ell} dr\bigg)\cdot \epsilon^{\ell-1}\\\nonumber
\le{} &
C(\ell,K)\big(1+e^{\lambda T}EV(X_{k,0},\mathcal L_{X_{k,0}})\big)\cdot \epsilon^{\ell}.
\end{align}
Let $n=\left[\frac{T}{\epsilon}\right]+1$ where $[m]$ denotes the integer part of $m\in \mathbb R$. Then we get
\begin{align*}
&E\big(\sup_{s,t\in[0,T],|t-s|\leq \epsilon}| X_{k,t}- X_{k,s}|^{2\ell}\big)\\
\leq{} &
C(\ell)\sum_{j=1}^{n} E\big(\sup_{t\in[(j-1)\epsilon,j\epsilon]}| X_{k,t}-X_{k,(j-1)\epsilon}|^{2\ell}\big)\\
\leq{} &
C(\ell,K)\big(1+e^{\lambda T}EV(X_{k,0},\mathcal L_{X_{k,0}})\big)\cdot(T+\epsilon)\epsilon^{\ell-1}.
\end{align*}
Thus, by Ascoli-Arzela theorem, $\{\mathcal L_{X_k}\}$ is tight on $C[0,T]$,
i.e. there exists a subsequence, still denoted $\{\mathcal L_{X_k}\}$,
which is weakly convergent to some measure $\mu\in\mathcal P(C[0,T])$.
By Skorokhod's representation theorem, there are $C[0,T]$-valued random variables $\bar {X}_k$ and $\bar X$
on some probability space $(\bar \Omega,\bar {\mathcal F}, \bar P)$
such that $\mathcal L_{\bar {X}_k}=\mathcal L_{X_k}, \mathcal L_{\bar X}=\mu$, and
\begin{align}\label{Limit}
\bar {X}_k\to \bar X~a.s.
\end{align}
Note that we have
\begin{align*}
\bar EV(\bar {X}_{k,t},\mathcal L_{\bar {X}_{k,t}})
\le&{}
e^{\lambda t}\bar EV(\bar {X}_{k,0},\mathcal L_{\bar {X}_{k,0}}),\\
\bar EV(\bar {X}_t,\mathcal L_{\bar {X}_t})
\le&{}
e^{\lambda t}\bar EV(\bar {X}_{0},\mathcal L_{\bar {X}_{0}}).
\end{align*}
And by H\"older's inequality, BDG's inequality and (H4), we obtain
\begin{align*}
\bar E\sup_{t\in [0,T]}| \bar {X}_{k,t}|^{2\ell}
={} &
E\sup_{t\in [0,T]}|X_{k,t}|^{2\ell}\\
\le{} &
C(\ell)E(\sup_{t\in [0,T]}| X_{k,t}- X_{k,0}|^{2\ell})+ C(\ell)E|X_{k,0}|^{2\ell}\\
={} &
C(\ell)E\bigg(\sup_{t\in [0,T]}\bigg|\int_0^t b_k(r,X_{k,r},\mathcal{L}_{X_{k,r}})dr\\
& +{} \int_0^t \sigma_k(r,X_{k,r},\mathcal{L}_{X_{k,r}})dW_r \bigg|^{2\ell}\bigg)+ C(\ell)E|X_{k,0}|^{2\ell}\\\nonumber
\leq{} &
C(\ell) E\int_0^T| b_k(r,X_{k,r},\mathcal{L}_{X_{k,r}})|^{2\ell}dr\cdot T^{2\ell-1}\\\nonumber
& +{}
C(\ell)E\bigg(\int_0^T|\sigma_k(r,X_{k,r},\mathcal{L}_{X_{k,r}})|^2 dr\bigg)^{\ell}+ C(\ell)E|X_{k,0}|^{2\ell}\\\nonumber
\leq{} &
C(K,\ell) E\int_0^T1+V(X_{k,r},\mathcal{L}_{X_{k,r}})dr\cdot T^{2\ell-1}\\\nonumber
& +{}
C(\ell)E\int_0^T|\sigma_k(r,X_{k,r},\mathcal{L}_{X_{k,r}})|^{2\ell} dr\cdot T^{\ell-1} + C(\ell)E|X_{k,0}|^{2\ell}\\\nonumber
\le{} &
C(K,\ell)\big(1+e^{\lambda T}EV(X_{k,0},\mathcal L_{X_{k,0}})\big)\cdot (T^{2\ell}+T^{\ell})+ C(\ell)E|X_{k,0}|^{2\ell}\\
<{} &
\infty.
\end{align*}
Thus, by limit \eqref{Cond1} and Vitali convergence theorem we get
\begin{align}\label{ESX_k}
\lim_{k\to\infty}\bar E\sup_{t\in [0,T]}|\bar X_{k,t}-\bar X_t|^{2}=0.
\end{align}
If $\bar X$ is a weak solution of equation \eqref{MVSDE}, by the weak uniqueness we have $\mathcal L_{X}=\mathcal L_{\bar X}$ on $C[0,T]$.
Therefore we get
\begin{align}\nonumber
\lim_{k\to\infty}\bar W_2(\mathcal L_{X_k},\mathcal L_X)
=\lim_{k\to\infty}\bar W_2(\mathcal L_{\bar X_k},\mathcal L_{\bar X})
\le \lim_{k\to\infty}\bar E\sup_{t\in [0,T]}|\bar X_{k,t}-\bar X_t|^2
=0,
\end{align}
and the proof is complete.

Now we prove that $\bar X$ is a weak solution for equation \eqref{MVSDE}.
Let $\bar {M}_{k,t}:= \bar {X}_{k,t}- \bar {X}_{k,0}- \int_0^t b_k(r,\bar {X}_{k,r},\mathcal{L}_{\bar {X}_{k,r}})dr$.
By the identical distribution between $\bar {X}_k$ and $X_k$, we have
\begin{align}\label{Mar1}
&\bar E[(\bar {M}_{k,t}-\bar {M}_{k,s})\cdot \Psi(\bar {X}_k|_{[0,s]})]\\\nonumber
= {}&
\bar E\bigg[\bigg(\bar {X}_{k,t}- \bar {X}_{k,s}- \int_s^t b_k(r,\bar {X}_{k,r},\mathcal{L}_{\bar {X}_{k,r}})dr\bigg)\cdot \Psi(\bar {X}_k|_{[0,s]})\bigg]\\\nonumber
= {}&
0,
\end{align}
and
\begin{align}\label{Mar2}
&\bar E\bigg[\bigg(\bar {M}_{k,t}\cdot\bar {M}_{k,t}-\bar {M}_{k,s}\cdot\bar {M}_{k,s}\\\nonumber
&\quad - \int_s^t (\sigma_k\cdot \sigma_k^{\top})(r,\bar {X}_{k,r},\mathcal{L}_{\bar {X}_{k,r}})dr\bigg)\cdot \Psi(\bar {X}_k|_{[0,s]})\bigg]
=0,
\end{align}
for any bounded continuous function $\Psi$ on $C[0,T]$.
That is to say that $\bar {M}_{k}$ is a square-integrable martingale, and its quadratic variation
\begin{equation}\nonumber
\left[\bar {M}_{k}\right]_t= \int_0^t (\sigma_k\cdot \sigma_k^{\top})(r,\bar {X}_{k,r},\mathcal{L}_{\bar {X}_{k,r}})dr.
\end{equation}
By the continuity of coefficient $b_k,b$, we get
\begin{align}\nonumber
\lim_{n\to\infty}|b_k(r,\bar {X}_{k,r},\mathcal{L}_{\bar {X}_{k,r}})-b(r,\bar X_r,\mathcal{L}_{\bar X_r})|=0
\quad \bar P\raisebox{0mm}{-} a.s.
\end{align}
and due to (H4), we have
\begin{align*}
&\bar E\int_0^t|b_k(r,\bar {X}_{k,r},\mathcal{L}_{\bar {X}_{k,r}})-b(r,\bar X_r,\mathcal{L}_{\bar X_r})|^{2\ell}dr\\
\le{} &
C(\ell)\bar E\int_0^t|b_k(r,\bar {X}_{k,r},\mathcal{L}_{\bar {X}_{k,r}})|^{2\ell}
+ |b(r,\bar X_r,\mathcal{L}_{\bar X_r})|^{2\ell}dr\\
\le{} &
C(\ell,K)\int_0^t 1+ \bar EV(\bar {X}_{k,r},\mathcal{L}_{\bar {X}_{k,r}})+\bar EV(\bar X_r,\mathcal{L}_{\bar X_r})dr\\
< {}&\infty.
\end{align*}
Therefore, by Vitali convergence theorem, we obtain
\begin{align}\nonumber
\lim_{n\to\infty}\bar E\int_0^t|b_k(r,\bar {X}_{k,r},\mathcal{L}_{\bar {X}_{k,r}})-b(r,\bar X_r,\mathcal{L}_{\bar X_r})|^{2}dr=0.
\end{align}
In a similar way, we have
\begin{align}\nonumber
\lim_{n\to\infty}\bar E\int_0^t|\sigma_k(r,\bar {X}_{k,r},\mathcal{L}_{\bar {X}_{k,r}})-\sigma(r,\bar X_r,\mathcal{L}_{\bar X_r})|^{2}dr=0.
\end{align}
Therefore, by limit \eqref{ESX_k} let $k\to\infty$ in equality \eqref{Mar1} as well as \eqref{Mar2} and denote $\bar M:=\lim_{k\to\infty}\bar {M}_k$,
we get
\begin{align*}
&\bar E[(\bar M_t-\bar M_s)\cdot \Psi(\bar X|_{[0,s]})]=0,\\
&\bar E\big[\big(\bar M_t\cdot\bar M_t-\bar M_s\cdot\bar M_s- \int_s^t\sigma\cdot \sigma^{\top}(r,\bar X_r,\mathcal{L}_{\bar X_r})dr\big)\cdot \Psi(\bar X|_{[0,s]})\big]
=0.
\end{align*}
By martingale limit theorem (see \cite[Proposition 1.3]{Chung-Williams} for instance),
$\bar M$ is a square integrable martingale, and its quadratic variation
$$
\left[\bar M\right]_t= \int_0^t \sigma\cdot \sigma^{\top}(r,\bar X_r,\mathcal{L}_{\bar X_r})dr.
$$
Therefore, by martingale representation theorem, there exists a Brownian motion $\bar W$
on the probability space $(\bar \Omega,\bar {\mathcal F}, \bar P)$ such that
$$
\bar M_t= \int_0^t \sigma(r,\bar X_r,\mathcal{L}_{\bar X_r})d\bar W_r.
$$
Thus, let $k\to\infty$ in the following SDE:
\begin{align}\label{SMVSDE1}
\left\{
\begin{aligned}
d\bar X_{k,t}
={} &
b_k(t,\bar X_{k,t},\mathcal{L}_{\bar X_{k,t}})dt
 +
d\bar M_{k,t}\\
\bar X_{k,0} = {}& \bar\xi_{k}, \\
\end{aligned}
\right.
\end{align}
we obtain that $(\bar X, \bar W)$ is a weak solution for MVSDE \eqref{MVSDE}.
\end{proof}

\subsection{Continuous dependence of invariant measures}
In this subsection, we investigate the continuous dependence of invariant measures on parameters under distribution-dependent Lyapunov conditions for the following MVSDEs:
\begin{align}\label{EQPIn}
dX_{k,t}= b_k(X_{k,t},\mathcal{L}_{X_{k,t}})dt+ \sigma_k(X_{k,t},\mathcal{L}_{X_{k,t}})dW_t,
\end{align}
where $k=1,2,...$,
\begin{align}\label{EQIn}
dX_t= b(X_t,\mathcal{L}_{X_t})dt+ \sigma(X_t,\mathcal{L}_{X_t})dW_t,
\end{align}
which is obtained by Skorokhod's representation theorem and martingale representation theorem.
Moreover, in this subsection, we do not need condition `$X_{k,0}\to X_0$ in some sense'.
\begin{thm}\label{thm2}
Assume that the conditions of Proposition \ref{prop2} hold with $b,\sigma$ replaced by $b_k,\sigma_k$.
Assume further that
\begin{align}\label{CoeCon1}
\lim_{k\to \infty}|b_k(x, \mu)- b(x, \mu)|+ |\sigma_k(x, \mu)- \sigma(x, \mu)|=0 \quad  \mbox{point-wise}.
\end{align}
Then we have
$$\mu_k\to \mu$$
where $\mu_k, \mu$ denote invariant measures for equations \eqref{EQPIn} and \eqref{EQIn}, respectively.
\end{thm}
\begin{proof}
By It\^o's formula and (H6), we have
\begin{align*}
EV(X_{k,t}, \mathcal L_{X_{k,t}})
& ={}
EV(X_{k,0}, \mathcal L_{X_{k,0}})+ E\int_0^t(\mathcal LV)(X_{k,s}, \mathcal L_{X_{k,s}})ds\\
& \le{}
EV(X_{k,0}, \mathcal L_{X_{k,0}})- \gamma t,
\end{align*}
where $X_{k}$ denotes the solution of the following MVSDE:
\begin{align}
\left\{
\begin{aligned}\label{pMVSDE1}
dX_{k,t}&= b_k(X_{k,t},\mathcal{L}_{X_{k,t}})dt+ \sigma_k(X_{k,t},\mathcal{L}_{X_{k,t}})dW_t,\\
X_{k,0}&= \xi_k.
\end{aligned}
\right.
\end{align}
By Gronwall's inequality, we get
$$EV(X_{k,t}, \mathcal L_{X_{k,t}})
\le
EV(X_{k,0}, \mathcal L_{X_{k,0}})e^{- \gamma t}
\le
EV(X_{k,0}, \mathcal L_{X_{k,0}})$$
for any $t\ge 0$.
Note that we have
\begin{align*}
EV(X_{k,t}, \mathcal L_{X_{k,t}})
\ge&
P(V(X_{k,t},\mathcal L_{X_{k,t}})>R)\cdot R\\
=&
P(|X_{k,t}|\vee |\mathcal L_{X_{k,t}}|_2>N)\cdot R
\ge
P(|X_{k,t}|>N)\cdot R,
\end{align*}
where constant $N$ only depends on $R$ with $N\to\infty$ as $R\to\infty$.
Thus, we obtain
$$P(|X_{k,t}|>N)
\le \frac{EV(X_{k,0}, \mathcal L_{X_{k,0}})}{R}.$$
That is to say that $\mathcal L_{X_{k,t}}$ is tight,
i.e. for any $\epsilon>0$, there exists a compact set $K_{\epsilon}\subset \mathbb R^d$ uniformly on $[0,\infty)$,
such that $\mathcal L_{X_{k,t}}(K_{\epsilon})>1-\epsilon$ for any $k$.
Applying Proposition \ref{prop2} and \cite[Theorem 2.1]{Billingsley}, we have
$$\mu_k(K_{\epsilon})
\ge \limsup_{t\to\infty}\mathcal L_{X_{k,t}}(K_{\epsilon})
> 1-\epsilon.$$
Therefore, $\{\mu_k\}$ is tight.
So there is a subsequence, still denoted $\{\mu_k\}$,
which is weakly convergent to some measure $\bar \mu$.
And similar to the proof of Theorem \ref{thm1}, we have $\bar\mu=\mu$.
Therefore, the proof is complete.
\end{proof}

\begin{rem}
By Theorem \ref{thm1}, it seems that we could obtain $\lim_{k\to\infty}W_2(\mu_k,\mu)=0.$
However, it may not hold since we can only get $X_{k}\to X ~in ~W_2~on~C[0,T]$ for any $T>0$.
If we can get the $X_{k}\to X ~in ~W_2~on~C[0,\infty)$, this result directly follows.
\end{rem}

\subsection{Counterexample}
For the classical SDEs,
if the initial value is convergent in probability,
the corresponding solution is convergent in probability uniformly on $[0,T]$.
However, this result does not hold for MVSDEs since the metric between measures does not match.
That is to say that the distance,
under which the coefficients are continuous with respect to the distribution variables,
does not match the distance between initial values.
Therefore, we utilize the property, which is that the convergence in probability does not imply the convergence in $L^2$,
to give a counterexample to illustrate this situation.
\begin{exam}
For any $k=1,2,...$, let $X_0, X_{k,0}$ be random variables satisfying $X_{k,0}\ge X_0, X_{k,0}\neq X_0~a.s.$,
and $X_{k,0}\to X_0$ in probability but not in $L^p$ for any $p>0$.
And denote $X_k, X$ by the solutions of the following MVSDEs with initial values $X_{k,0},X_0$, respectively:
\begin{align}
dX_{k,t}= \int_{\mathbb R^d}y\mathcal L_{{k,t}}(dy)dt+ \sqrt 2 dW_t,
\end{align}
and
\begin{align}
dX_t= \int_{\mathbb R^d}y\mathcal L_{X_t}(dy)dt+ \sqrt 2 dW_t.
\end{align}
Then $X_k$ does not converge to $X$ in probability.
\end{exam}

\begin{proof}
Now we construct two sequences of stochastic processes $\{X^{(n)}_k\}$ and $\{X^{(n)}\}$ as follows:
$$
X_{k,t}^{(0)}= X_{k,0}, X^{(0)}_t= X_{0}, ~t\geq 0,
$$
for any $n\geq 1$,
\begin{equation}
 \left\{
  \begin{aligned}
    &dX^{(n)}_{k,t}= \int_{\mathbb R^d}y\mathcal L_{X^{(n-1)}_{k,t}}(dy)dt+ \sqrt 2 dW_t,\\
    &X^{(n)}_{k,0}= X_{k,0},
\end{aligned}
 \right.
\end{equation}
and
\begin{equation}
 \left\{
  \begin{aligned}
    &dX^{(n)}_k= \int_{\mathbb R^d}y\mathcal L_{X^{(n-1)}_t}(dy)dt+ \sqrt 2 dW_t,\\
    &X^{(n)}_0= X_{0},
\end{aligned}
 \right.
\end{equation}
By Wang \cite[Lemma 2.3]{Wang_18}, we have
\begin{align}\label{SolConW1}
\lim_{n\to \infty}E\sup_{0\leq t\leq T}|X_{k,t}- X^{(n)}_{k,t}|^2=0,\\ \nonumber
\lim_{n\to \infty}E\sup_{0\leq t\leq T}|X_t- X^{(n)}_t|^2=0.
\end{align}
By the definition of $\{X^{(n)}_k\}$ and $\{X^{(n)}\}$, we get
\begin{align*}
X^{(1)}_{k,t}- X^{(1)}_t&= X_{k,0}-X_0 + \int_0^t E(X^{(0)}_{k,s}-X^{(0)}_s)ds= X_{k,0}-X_0+ E(X_{k,0}-X_0)t,\\
X^{(2)}_{k,t}- X^{(2)}_t&= X_{k,0}-X_0 + \int_0^t E(X^{(1)}_{k,s}- X^{(1)}_s)ds\\
                         &= X_{k,0}-X_0+ \int_0^t (E(X_{k,0}-X_0)+ E(X_{k,0}-X_0)s) ds\\
                         &= X_{k,0}-X_0+ E(X_{k,0}-X_0)t+ \frac 12E(X_{k,0}-X_0) t^2,\\
...\\
X^{(n)}_{k,t}- X^{(n)}_t&= X_{k,0}-X_0+ \int_0^t E(X^{(n-1)}_{k,s}-X^{(n-1)}_s)ds\\
            &= X_{k,0}-X_0+ E(X_{k,0}-X_0)t+ \frac 12 E(X_{k,0}-X_0)t^2+... + \frac {1}{n!} E(X_{k,0}-X_0)t^n.
\end{align*}
Note that we have
\begin{align}\label{INeP}
P(|X^{(n)}_k- X^{(n)}|>\epsilon)
\leq &
P(|X^{(n)}_k- X_k|> \frac \epsilon3)+ P(|X_k- X|>\frac \epsilon3)
+
P(|X- X^{(n)}|>\frac \epsilon3),
\end{align}
for any $\epsilon>0$.
Since the convergence in $L^p$ implies the convergence in probability,
by two equalities in the right side of \eqref{SolConW1} we obtain
\begin{align*}
\lim_{n\to \infty}P(\sup_{0\le t\le T}|X^{(n)}_{k,t}- X_{k,t}|> \frac \epsilon3)=0,\\
\lim_{n\to \infty}P(\sup_{0\le t\le T}|X_t- X^{(n)}_t|>\frac \epsilon3)=0.
\end{align*}
Let $n\to \infty$ then $k\to \infty$ in inequality \eqref{INeP},
we get
\begin{align}\nonumber
\lim_{k\to \infty}P(|X_{k,t}- X_t|>\frac \epsilon3)
& \geq
\lim_{k\to \infty}\lim_{n\to \infty}P(|X^{(n)}_{k,t}- X^{(n)}_t|>\epsilon)=1,
\end{align}
where the last equality holds for some sufficiently small constant $\epsilon>0$ and $t\neq 0$.
The proof is complete.
\end{proof}

However convergence in probability for initial values does not
imply that the corresponding solutions converge in probability under Lipschitz condition by the above counterexample,
we obtain that solutions inherit the $L^2$-convergence of initial values,
which is showed by the following proposition.
And we think that the proof is simple, so we put it in Appendix.
\begin{prop}\label{Thm1}
Suppose that the coefficients $b,\sigma,b_k,\sigma_k$ of equations \eqref{MVSDE} and \eqref{pMVSDE} are Lipschitz with Lipschitz constant $L$,
and satisfy the linear growth conditions with linear growth constant $K$ for all $k=1,2,...$.
Assume further that
\begin{align}\label{IniCon1}
\lim_{k\to \infty}E(|X_{k,0}-X_0|^2)=0,
\end{align}
and
\begin{align}\label{CoeCon1}
\lim_{k\to \infty}|b_k(t, x, \mu)- b(t, x, \mu)|+ |\sigma_k(t, x, \mu)- \sigma(t, x, \mu)|=0 \quad  \mbox{point-wise}.
\end{align}
Then we have
\begin{align}\label{1Con_2}
\lim_{k\to \infty}E\sup_{0\leq t\leq T}|X_{k,t}- X_t|^2=0.
\end{align}
\end{prop}

\begin{rem}
By the above Proposition, we directly obtain \cite[Proposition 5.2]{Qiao} and \cite[Theorem 5.1 and Theorem 6.1]{BMM}.
\end{rem}

As a direct consequence of Proposition \ref{Thm1}, we have the following corollary.
\begin{cor}\label{Cor1}
Assume that the conditions of Proposition \ref{Thm1} hold with \eqref{IniCon1} replaced by
\begin{align}\nonumber
\lim_{k\to\infty}W_2(\mathcal L_{X_{k,0}}, \mathcal L_{X_0})=0.
\end{align}
Then we have
\begin{align}\nonumber
\lim_{k\to\infty}\bar W_2(\mathcal L_{X_k}, \mathcal L_{X})=0.
\end{align}
\end{cor}

\begin{proof}
For each $k\ge1$, there exist random variables $\tilde X_{k,0}, \tilde X_0^{(k)}$ such that
$\mathcal L_{X_{k,0}}= \mathcal L_{\tilde X_{k,0}}$, $\mathcal L_{X_0}= \mathcal L_{\tilde X_0^{(k)}}$ and
\begin{align}\nonumber
W_2(\mathcal L_{X_{k,0}}, \mathcal L_{X_0})^2= E|\tilde X_{k,0}- \tilde X_0^{(k)}|^2.
\end{align}
Let $\tilde X_{k,t},\tilde X^{k}_t$ be the solutions of MVSDEs \eqref{MVSDE} and \eqref{pMVSDE} with initial value $\tilde X_{k,0},\tilde X_{0}^k$ respectively.
By Proposition \ref{Thm1}, we have
\begin{align}\nonumber
\lim_{k\to\infty}E\sup_{0\leq s\leq T}|\tilde X_{k,s}- \tilde X^{(k)}_s|^2=0.
\end{align}
Therefore, by the weak uniqueness we obtain
\begin{align}\nonumber
\bar W_2(\mathcal L_{X_k}, \mathcal L_{X})^2
=
\bar W_2(\mathcal L_{\tilde X_k}, \mathcal L_{\tilde X^{(k)}})^2
\le
E\sup_{0\leq t\leq T}|\tilde X_{k,t}-\tilde X^{(k)}_t|^2\to0.
\end{align}
The proof is complete.
\end{proof}

\begin{rem}
We know that the optimal coupling is attained, but the optimal one depend on $k$
and we could not find a random variable $\tilde X_0$
such that $W_2(\mathcal L_{X_{k,0}}, \mathcal L_{X_0})^2= E|\tilde X_{k,0}- \tilde X_0|^2$ holds for every $k=1,2,...$
\end{rem}

\begin{exam}
(Small `distribution' limit).
In particular, the random trajectory of the MVSDE
\begin{equation*}
   \left\{
   \begin{aligned}
   &dX^{\epsilon}_t= b_{\epsilon}(X^{\epsilon}_t, \mathcal L_{X^{\epsilon}_t})dt+ \sigma_{\epsilon}(X^{\epsilon}_t, \mathcal L_{X^{\epsilon}_t})dW_t\\
   &X^{\epsilon}_0=\xi
   \end{aligned}
   \right.
\end{equation*}
converge in $L^2$ uniformly on $[0,T]$ as $\epsilon \to 0$ to the trajectory of
\begin{equation*}
   \left\{
   \begin{aligned}
   &dX_t= b(X_t)dt+\sigma(X_t)dW_t\\
   &X_0=\xi,
   \end{aligned}
   \right.
\end{equation*}
where $\lim_{\epsilon \to 0}b_{\epsilon}(x,\mu)=b(x),\lim_{\epsilon \to 0}\sigma_{\epsilon}(x,\mu)=\sigma(x)$ for each $x\in\mathbb R^d, \mu\in\mathcal P(\mathbb R^d)$.
\end{exam}

We next give the continuous dependence of invariant measures on parameters without the dependence on initial values.
The proof is also put in Appendix.
\begin{prop}\label{Thm3}
Suppose that the conditions of Proposition \ref{Thm1} hold except for the condition \eqref{IniCon1},
and for the Lipschitz conditions of coefficients $b,b_k$ replaced by the following conditions:
there exist nonnegative constants $L_1,L_2$  with $2L_1-2L_2- 8L^2>0$ such that
\begin{align}\nonumber
\langle x-y, b(x,\mu)- b(y,\nu)\rangle\leq -L_1|x-y|^2+ L_2|x-y|W_2(\mu,\nu),
\end{align}
for any $x,y\in\mathbb R^d,\mu,\nu\in\mathcal P_2(\mathbb R^d)$.
Assume further that the coefficients $b,b_k,\sigma,\sigma_k$ are independent of $t$.
Then we have
\begin{align}\label{3Con_1}
\lim_{k\to \infty}W_2(\mu_k,\mu)=0,
\end{align}
where $\mu_k,\mu$ are the unique invariant measures of MVSDEs \eqref{EQPIn} and \eqref{EQIn}, respevtively.
\end{prop}

\section{Application}
In this section, we give some examples to illustrate our theoretical results.
\begin{exam}
Consider the McKean-Vlasov SDEs
\begin{align*}
dX_{\lambda,t}
& = -\lambda X_{\lambda,t}\int_{\mathbb R}y^2\mathcal L_{X_{\lambda,t}}(dy)dt+(\sqrt 2+\lambda)X_{\lambda,t}dW_t\\
& =:
b_{\lambda}(X_{\lambda,t},\mathcal L_{X_{\lambda,t}})dt+\sigma_{\lambda}(X_{\lambda,t},\mathcal L_{X_{\lambda,t}})dW_t,
\end{align*}
where $\lambda\in [0,1]$ is a parameter and $W$ is a one-dimensional Brownian motion.
Then the solution $X_{\lambda,t}$ with initial value $X_{\lambda,0}$ for the above MVSDEs converges to $X_t$,
which is the solution of the following SDEs with initial value $X_0$, in $W_2$ uniformly on $[0,T]$
\begin{align*}
dX_{t}
& = - X_{t}\int_{\mathbb R}y^2\mathcal L_{X_{t}}(dy)dt+(\sqrt 2+1)X_{t}dW_t\\
& = b(X_t,\mathcal{L}_{X_t})dt+ \sigma(X_t,\mathcal{L}_{X_t})dW_t.
\end{align*}
\end{exam}
\begin{proof}
(H1), (H3) and (H5) obviously hold. Now we prove that (H2) and (H4) hold.
For any $x\in\mathbb R, \mu\in \mathcal P(\mathbb R)$, let $V(x,\mu)=x^{6}+\int_{\mathbb R}y^{10}\mu(dy)$. Then (H4) holds and
\begin{align*}
\partial _xV(x,\mu)&={}6x^5, ~ \partial _x^2V(x,\mu)={}30x^4,\\
\partial _{\mu}V(x,\mu)(z)&={}10z^9, ~ \partial _z\partial _{\mu}V(x,\mu)(z)=90z^8.
\end{align*}
Thus, we have
\begin{align*}
(\mathcal LV)(x,\mu)
={} &
-\lambda x\int y^2\mu(dy)\cdot 6x^5+ \frac 12 \cdot (\sqrt2+\lambda)^2x^2\cdot 30x^4\\
& +
\int\bigg((-\lambda z\int y^2\mu(dy)) \cdot 10z^9+\frac 12 \cdot (\sqrt2+\lambda)^2z^2\cdot 90z^8 \bigg)\mu(dz)\\
\le{} &
90x^6+ \int 270z^6 \mu(dz)
\le 270 V(x,\mu),
\end{align*}
i.e. (H2) holds.
Therefore, by Theorem \ref{thm1}, we get the desired result.
\end{proof}

\begin{exam}
For any $x\in\mathbb R, \mu\in \mathcal P(\mathbb R)$, let
\begin{align*}
b_{\lambda}(x,\mu)&=-3\lambda x+3\lambda\int_{\mathbb R} y\mu(dy), ~\sigma_{\lambda}(x,\mu)= x-\int_{\mathbb R} y\mu(dy),\\
b(x,\mu)&=-3x+3\int_{\mathbb R} y\mu(dy), ~\sigma(x,\mu)= x-\int_{\mathbb R} y\mu(dy),
\end{align*}
where $\lambda$ is a parameter in $[1,2]$.
Then $\mu_{\lambda}\to\mu$ weakly as $\lambda\to 1$, where $\mu_{\lambda}$ and $\mu$ are invariant measures of following MVSDEs:
\begin{align*}
dX_{\lambda,t}&= b_{\lambda}(X_{\lambda,t},\mathcal{L}_{X_{\lambda,t}})dt+ \sigma_{\lambda}(X_{\lambda,t},\mathcal{L}_{X_{\lambda,t}})dW_t,\\
dX_t&= b(X_t,\mathcal{L}_{X_t})dt+ \sigma(X_t,\mathcal{L}_{X_t})dW_t.
\end{align*}
\end{exam}

\begin{proof}
Conditions (H1), (H3) and (H5) obviously hold.
Let $V(x,\mu)=\frac 14\big(x-\int_{\mathbb R} y\mu(dy)\big)^4$. Then it is immediately to see that (H4) holds. Note that we have
\begin{align*}
\partial _xV(x,\mu)&={}\big(x-\int y\mu(dy)\big)^3, ~ \partial _x^2V(x,\mu)={}3\big(x-\int y\mu(dy)\big)^2,\\
\partial _{\mu}V(x,\mu)(z)&={}-\big(x-\int y\mu(dy)\big)^3, ~ \partial _z\partial _{\mu}V(x,\mu)(z)=0,
\end{align*}
and
\begin{align*}
(\mathcal LV)(x,\mu)
=&{}
-3\lambda\big(x-\int y\mu(dy)\big)^4+\frac 32\big(x-\int y\mu(dy)\big)^4\\
&+{}
\int 3\big(z-\int y\mu(dy)\big)\big(x-\int y\mu(dy)\big)^3\mu(dz)\\
=&{}
(-3\lambda+\frac 32)\big(x-\int y\mu(dy)\big)^4
\le
-\frac 32V(x,\mu).
\end{align*}
That is to say that (H2) holds.
Therefore, by Theorem \ref{thm2}, we obtain the desired results.
\end{proof}
\begin{rem}
It is easy to verify that the coefficients $b_{\lambda},\sigma_{\lambda},b,\sigma$ do not satisfy the conditions of Proposition \ref{Thm3},
so here we apply Theorem \ref{thm2} to illustrate the result.
\end{rem}

\begin{exam}
Consider the McKean-Vlasov SDE
\begin{align*}
dX_{\lambda,t}
& =
\big((-6+\lambda)X_{\lambda,t}+\int_{\mathbb R}y\mathcal L_{X_{\lambda,t}}(dy)\big)dt
+\big(X_{\lambda,t}+\lambda\int_{\mathbb R}y\mathcal L_{X_{\lambda,t}}(dy)\big)dW(t)\\
& =:
b_{\lambda}(X_{\lambda,t},\mathcal L_{X_{\lambda,t}})dt+\sigma_{\lambda}(X_{\lambda,t},\mathcal L_{X_{\lambda,t}})dW_t,
\end{align*}
where $\lambda\in[0,1)$ is a parameter and $W$ is a one-dimensional Brownian motion.
Clearly $b_{\lambda}$ and $\sigma_{\lambda}$ are Lipschitz and satisfy linear growth condition.
We note that for any $x\in\mathbb R, \mu\in \mathcal P(\mathbb R)$,
\begin{align*}
\lim_{\lambda\to\lambda_0}|b_{\lambda}(x,\mu)-b_{\lambda_0}(x,\mu)|+|\sigma_{\lambda}(x,\mu)-\sigma_{\lambda_0}(x,\mu)|=0.
\end{align*}
Then, by Proposition \ref{Thm1} and Corollary \ref{Cor1}, the unique solution $X_{\lambda}$ converges to $X_{\lambda_0}$
in $L^2$ uniformly on $[0,T]$,
and the corresponding measure $\mathcal L_{X_{\lambda}}$ converges to $\mathcal L_{X_{\lambda_0}}$ under $\bar W_2$.

Moreover, we have
\begin{align*}
\langle x-y, b_{\lambda}(x,\mu)-b_{\lambda}(y,\nu)\rangle
= &
\langle x-y,(-6+\lambda)(x-y)+ \int z\mu(dz)-\int z\nu(dz)\rangle\\
\le &
(-6+\lambda)|x-y|^2+|x-y|W_2(\mu,\nu),\\
|\sigma_{\lambda}(x,\mu)-\sigma_{\lambda}(y,\nu)|
= &
|x-y+\lambda\int z\mu(dz)-\lambda\int z\nu(dz)|\\
\le &
|x-y|+\lambda W_2(\mu,\nu).
\end{align*}
Therefore, $L_1=6-\lambda,L_2=L=1$, and $2L_1-2L_2-8L^2= 2(6-\lambda)-2-8= 2(1-\lambda)>0$.
Applying Proposition \ref{Thm3}, we obtain that the uniqueness invariant measure $\mu_{\lambda,I}$ converges to $\mu_{\lambda_0,I}$ under $W_2$.
\end{exam}

\section*{Appendix}
In this section, we give the proof of Proposition \ref{Thm1} and Proposition \ref{Thm3}.
\begin{proof}[Proof of Proposition \ref{Thm1}]
Note that we have
\begin{align*}
X_{k,t}-X_t
= &
X_{k,0}-X_0 + \int_0^t b_k(s,X_{k,s},\mathcal{L}_{{k,s}})- b(s,X_s,\mathcal{L}_{X_s})ds\\
& +
\int_0^t \sigma_k(s,X_{k,s},\mathcal{L}_{X_{k,s}})- \sigma(s,X_s,\mathcal{L}_{X_s})dW_s\\
= &
X_{k,0}-X_0 + \int_0^t b_k(s,X_s,\mathcal{L}_{X_s})- b(s,X_s,\mathcal{L}_{X_s})ds\\
&+
\int_0^t \sigma_k(s,X_s,\mathcal{L}_{X_s})- \sigma(s,X_s,\mathcal{L}_{X_s})dW_s\\
& +
\int_0^t b_k(s,X_{k,s},\mathcal{L}_{X_{k,s}})- b_k(s,X_s,\mathcal{L}_{X_s})ds\\
&+
\int_0^t \sigma_k(s,X_{k,s},\mathcal{L}_{X_{k,s}})- \sigma_k(s,X_s,\mathcal{L}_{X_s})dW_s\\
=: &
\eta_{k,t} + \int_0^t b_k(s,X_{k,s},\mathcal{L}_{X_{k,s}})- b_k(s,X_s,\mathcal{L}_{X_s})ds\\
&+
\int_0^t \sigma_k(s,X_{k,s},\mathcal{L}_{X_{k,s}})- \sigma_k(s,X_s,\mathcal{L}_{X_s})dW_s,
\end{align*}
where
\begin{align*}
\eta_{k,t}:=&X_{k,0}-X_0 + \int_0^t b_k(s,X_s,\mathcal{L}_{X_s})- b(s,X_s,\mathcal{L}_{X_s})ds\\
&+
\int_0^t \sigma_k(s,X_s,\mathcal{L}_{X_s})- \sigma(s,X_s,\mathcal{L}_{X_s})dW_s.
\end{align*}
By Cauchy-Schwarz inequality and martingale inequality, we get
\begin{align*}
&E\big(\sup_{0\leq s\leq t}|X_{k,s}-X_s|^2\big)\\
\leq{} &
3E\big(\sup_{0\leq s\leq t}|\eta_{k,s}|^2\big)
 +{}
 3E\bigg(\sup_{0\leq s\leq t}\bigg|\int_0^s b_k(r,X_{k,r},\mathcal{L}_{X_{k,r}})- b_k(r,X_r,\mathcal{L}_{X_r})dr\bigg|^2\bigg)\\
 &+{}
 3E\bigg(\sup_{0\leq s\leq t}\bigg|\int_0^s \sigma_k(r,X_{k,r},\mathcal{L}_{X_{k,r}})- \sigma_k(r,X_r,\mathcal{L}_{X_r})dW_r\bigg|^2\bigg)\\
\leq{} &
3E\big(\sup_{0\leq s\leq t}|\eta_{k,s}|^2\big)
 +{}
 3TE\bigg(\int_0^t |b_k(r,X_{k,r},\mathcal{L}_{X_{k,r}})- b_k(r,X_r,\mathcal{L}_{X_r})|^2 dr\bigg)\\
 &+{}
 3\times 4 E\bigg(\int_0^t |\sigma_k(r,X_{k,r},\mathcal{L}_{X_{k,r}})- \sigma_k(r,X_r,\mathcal{L}_{X_r})|^2 dr\bigg)\\
\leq{} &
3E\big(\sup_{0\leq s\leq t}|\eta_{k,s}|^2\big)
 +{}
 3TE\int_0^t 2L^2\big(|X_{k,r}-X_r|^2+ W_2(\mathcal{L}_{X_{k,r}},\mathcal{L}_{X_r})^2 \big)dr\\
 &+{}
 3\times 4E\int_0^t 2L^2\big(|X_{k,r}-X_r|^2+ W_2(\mathcal{L}_{X_{k,r}},\mathcal{L}_{{k,r}})^2 \big)dr\\
\leq{} &
3E(\sup_{0\leq s\leq t}|\eta_{k,s}|^2)
 +{}
 2(6TL^2+ 24L^2)\int_0^t E\big(\sup_{0\leq s\leq r}|X_{k,s}-X_s|^2\big)dr,
\end{align*}
for any $t\in [0,T]$.
Applying Gronwall's inequality, we obtain
$$
E\sup_{0\leq s\leq t}|X_{k,s}-X_s|^2
\leq
3E(\sup_{0\leq s\leq t}|\eta_{k,s}|^2)e^{2(6TL^2+ 24L^2)t}.
$$
So \eqref{1Con_2} holds if we can prove
$$
\lim_{k\to \infty}E(\sup_{0\leq t\leq T}|\eta_{k,t}|^2)=0.
$$

By Cauchy-Schwarz inequality and martingale inequality, we obtain
\begin{align*}
E(\sup_{0\leq t\leq T}|\eta_{k,t}|^2)
\leq{}&
3E|X_{k,0}-X_0|^2\\
&+{}
3E\bigg(\sup_{0\leq t\leq T}\bigg|\int_0^t b_k(s,X_s,\mathcal{L}_{X_s})- b(s,X_s,\mathcal{L}_{X_s})ds\bigg|^2\bigg)\\
&+{}
3E\bigg(\sup_{0\leq t\leq T}\bigg|\int_0^t \sigma_k(s,X_s,\mathcal{L}_{X_s})- \sigma(s,X_s,\mathcal{L}_{X_s})dW_s\bigg|^2\bigg)\\
\leq{} &
3E|X_{k,0}-X_0|^2\\
&+{}
3TE\int_0^T |b_k(s,X_s,\mathcal{L}_{X_s})- b(s,X_s,\mathcal{L}_{X_s})|^2 ds\\
&+{}
3\times 4E\int_0^T |\sigma_k(s,X_s,\mathcal{L}_{X_s})-\sigma(s,X_s,\mathcal{L}_{X_s})|^2 ds.
\end{align*}
We note that
\begin{align*}
|b_k(s,X_s,\mathcal{L}_{X_s})- b(s,X_s,\mathcal{L}_{X_s})|^2
\leq{} &
2|b_k(s,X_s,\mathcal{L}_{X_s})|^2+ 2|b(s,X_s,\mathcal{L}_{X_s})|^2\\
\leq{} &
12K^2\big(1+ |X_s|^2+ \mathcal{L}_{X_s}(|\cdot|^2)\big),
\end{align*}
and $E\sup_{0\le s\le T}|X_s|^2<\infty$ by \cite[Theorem 2.1]{Wang_18} or \cite[Theorem 3.1]{LM1}.
Applying Lebesgue dominated convergence theorem, we get
$$
\lim_{k\to \infty}E\bigg|\int_0^T b_k(s,X_s,\mathcal{L}_{X_s})- b(s,X_s,\mathcal{L}_{X_s})ds\bigg|^2=0.
$$
In the same way, we obtain
\begin{align*}
\lim_{k\to \infty}E\bigg|\int_0^T \sigma_k(s,X_s,\mathcal{L}_{X_s})- \sigma(s,X_s,\mathcal{L}_{X_s})dW_s\bigg|^2=0.
\end{align*}
Now $\lim_{k\to \infty}E\sup_{0\le t\le T}|\eta_{k,t}|^2=0$ holds and hence \eqref{1Con_2} is proved.
\end{proof}

\begin{proof}[Proof of Proposition \ref{Thm3}]
By \cite[Theorem 3.1]{Wang_18}, there exist invariant measures for MVSDEs \eqref{EQPIn} and \eqref{EQIn},
denoted by $\mu_k,\mu$, respectively.
Let $X_{k,0},X_0$ be random variables such that
$$
\lim_{k\to\infty}E|X_{k,0}-X_0|^2=0,
$$
Denote $X_{k,t},X_t$ by the corresponding solutions with initial values $X_{k,0},X_0$, respectively.
By It\^o's formula, we have
\begin{align*}
E|X_{k,t}- X_t|^2
= &
E|X_{k,0}-X_0|^2\\
 & +
 2E\int_0^t\langle X_{k,s}- X_s, b_k(X_{k,s}, \mathcal L_{X_{k,s}})-b(X_s, \mathcal L_{X_s})\rangle ds\\
 & +
 E\int_0^t|\sigma_k(X_{k,s}, \mathcal L_{X_{k,s}})-\sigma(X_s, \mathcal L_{X_s})|^2ds\\
\leq &
E|X_{k,0}-X_0|^2\\
 & +
 2E\int_0^t\langle X_{k,s}- X_s, b_k(X_{k,s}, \mathcal L_{X_{k,s}})-b_k(X_s, \mathcal L_{X_s})\rangle ds\\
 & +
 2E\int_0^t\langle X_{k,s}- X_s, b_k(X_s, \mathcal L_{X_s})-b(X_s, \mathcal L_{X_s})\rangle ds\\
 & +
 2E\int_0^t|\sigma_k(X_{k,s}, \mathcal L_{X_{k,s}})-\sigma_k(X_s, \mathcal L_{X_s})|^2ds\\
 & +
 2E\int_0^t|\sigma_k(X_s, \mathcal L_{X_s})-\sigma(X_s, \mathcal L_{X_s})|^2ds\\
\leq &
E|X_{k,0}-X_0|^2\\
 & + E\int_0^t -2L_1|X_{k,s}- X_s|^2+ 2L_2|X_{k,s}- X_s|W_2(\mathcal L_{X_{k,s}},\mathcal L_{X_s})ds\\
 & +
 4L^2 E\int_0^t |X_{k,s}- X_s|^2+ W_2(\mathcal L_{X_{k,s}},\mathcal L_{X_s})^2ds\\
 & +
 2E\int_0^t |X_{k,s}- X_s|\cdot |b_k(X_s, \mathcal L_{X_s})-b(X_s, \mathcal L_{X_s})|ds\\
 & +
 2E\int_0^t|\sigma_k(X_s, \mathcal L_{X_s})-\sigma(X_s, \mathcal L_{X_s})|^2ds\\
\leq &
E|X_{k,0}-X_0|^2+\int_0^t(-2L_1+ 2L_2+ 8L^2+ \epsilon)E|X_{k,s}- X_s|^2ds\\
 & +
 \frac 1\epsilon E\int_0^t|b_k(X_s, \mathcal L_{X_s})-b(X_s, \mathcal L_{X_s})|^2ds\\
 & +
 2E\int_0^t|\sigma_k(X_s, \mathcal L_{X_s})-\sigma(X_s, \mathcal L_{X_s})|^2ds,
\end{align*}
where $0<\epsilon< 2L_1- 2L_2- 8L^2$, and the last inequality holds owing to Cauchy-Schwarz inequality.
By Gronwall's inequality, we have
\begin{align}\label{IeqINC}
E|X_{k,t}- X_t|^2 \leq e^{(-2L_1+ 2L_2+ 8L^2+ \epsilon)t}E\eta_{k,t},
\end{align}
where
\begin{align*}
\eta_{k,t}:=&|X_{k,0}-X_0|^2+\frac 1\epsilon \int_0^t|b_k(X_s, \mathcal L_{X_s})-b(X_s, \mathcal L_{X_s})|^2ds\\
 & +
 2\int_0^t|\sigma_k(X_s, \mathcal L_{X_s})-\sigma(X_s, \mathcal L_{X_s})|^2ds.
\end{align*}
Applying \cite[Lemma 4.3]{Villani} we obtain
\begin{align*}
W_2(\mu_k,\mu)
\le
\liminf_{t\to \infty}W_2(\mathcal L_{X_{k,t}},\mathcal L_{X_t})
\le
\liminf_{t\to \infty}(E|X_{k,t}- X_t|^2)^{\frac 12}.
\end{align*}
If we can prove $\lim_{k\to\infty}\lim_{t\to\infty}e^{(-2L_1+ 2L_2+ 8L^2+ \epsilon)t}E\eta_{k,t}=0$, \eqref{3Con_1} holds.

By conditions on coefficients $b,\sigma$ with $y=0, \nu=\delta_0$, we get
\begin{align*}
\langle x,b(x,\mu)\rangle
&=
\langle x,b(x,\mu)-b(0,\delta_0)\rangle+ \langle x,b(0,\delta_0)\rangle\\
&\leq
-L_1|x|^2+ L_2x\cdot\big(\mu(|\cdot|^2)\big)^{\frac 12}+ |b(0,\delta_0)|\cdot|x|,\\
|\sigma(x,\mu)|^2
&\leq
2|\sigma(x,\mu)-\sigma(0,\delta_0)|^2+ 2|\sigma(0,\delta_0)|^2\\
&\leq
4L^2(|x|^2+\mu(|\cdot|^2))+ 2|\sigma(0,\delta_0)|^2.
\end{align*}
Applying It\^o's formula and Cauchy-Schwarz inequality, we obtain
\begin{align*}
E|X_t|^2
&=
E|X_0|^2+ 2E\int_0^t\langle X_s,b(X_s,\mathcal L_{X_s})\rangle ds+ E\int_0^t|\sigma(X_s,\mathcal L_{X_s})|^2 ds\\
&\leq
E|X_0|^2+ E\int_0^t (-2L_1+ 2L_2+ 8L^2+ \epsilon)|X_s|^2ds+ (\frac 1\epsilon |b(0,\delta_0)|^2+2|\sigma(0,\delta_0)|^2)t.
\end{align*}
By Gronwall's inequality, we have
$$
E|X_t|^2
\leq
\big(E|X_0|^2+(\frac 1\epsilon |b(0,\delta_0)|^2+2|\sigma(0,\delta_0)|^2)t\big) e^{(-2L_1+ 2L_2+ 8L^2+ \epsilon)t}.
$$
Thus, by linear growth condition of coefficients $b_k$ and $b$ we get
\begin{align*}
&E\int_0^t|b_k(X_s, \mathcal L_{X_s})-b(X(s), \mathcal L_{X(s)})|^2ds\\
\leq &
2E\int_0^t|b_k(X_s, \mathcal L_{X_s})|^2+ |b(X_s, \mathcal L_{X_s})|^2ds\\
\leq &
2E\int_0^t6K^2\big(1+ |X_s|^2+ \mathcal L_{X_s}(|\cdot|^2)\big)ds\\
\leq &
12K^2t+ 24K^2\int_0^t\big(E|X_0|^2+(\frac 1\epsilon |b(0,\delta_0)|^2+2|\sigma(0,\delta_0)|^2)s\big) e^{(-2L_1+ 2L_2+ 8L^2+ \epsilon)s}ds.
\end{align*}
Similarly, we obtain
\begin{align*}
&E\int_0^t|\sigma_k(X_s, \mathcal L_{X_s})-\sigma(X_s, \mathcal L_{X_s})|^2ds\\
\leq &
12K^2t+ 24K^2\int_0^t\big(E|X_0|^2+(\frac 1\epsilon |b(0,\delta_0)|^2+2|\sigma(0,\delta_0)|^2)s\big) e^{(-2L_1+ 2L_2+ 8L^2+ \epsilon)s}ds.
\end{align*}
Thus, we get $\lim_{k\to\infty}\lim_{t\to \infty}e^{(-2L_1+ 2L_2+ 8L^2+ \epsilon)t}E\eta_{k,t}=0$.
The proof is complete.
\end{proof}

\section*{Acknowledgement}
The first author is supported by
National Funded Postdoctoral Researcher Program (Grant GZC20230414)
and Fundamental Research Funds for the Central Universities (Grant 135114002).
The second author is
supported by National Key R\&D Program of China (No. 2023YFA1009200), NSFC (Grants
11871132, 11925102), LiaoNing Revitalization Talents Program (Grant XLYC2202042), and
Dalian High-level Talent Innovation Program (Grant 2020RD09).

\end{document}